\documentclass[12pt,reqno]{article}

\usepackage[usenames]{color}
\usepackage{graphics,amsmath,amssymb}
\usepackage{amscd}

\usepackage[colorlinks=true,
linkcolor=webgreen,
filecolor=webbrown,
citecolor=webgreen]{hyperref}

\definecolor{webgreen}{rgb}{0,.5,0}
\definecolor{webbrown}{rgb}{.6,0,0}

\usepackage{color}
\usepackage{fullpage}
\usepackage{float}

\usepackage{amsthm}
\usepackage{amsfonts}
\usepackage{latexsym}
\usepackage{epsf}

\usepackage{verbatim}
\usepackage{enumerate}

\usepackage{tikz}

\usepackage[blocks]{authblk}

\usepackage{algpseudocode}
\usepackage{multirow}

\theoremstyle{plain}
\newtheorem{theorem}{Theorem}

\newtheorem{lemma}[theorem]{Lemma}
\newtheorem{proposition}[theorem]{Proposition}

\theoremstyle{definition}
\newtheorem{definition}{Definition}
\newtheorem{example}{Example}

\theoremstyle{remark}

\begin{document}

\title{EvenQuads Game and Error-Correcting Codes}
\author{Nikhil Byrapuram}
\author{Hwiseo (Irene) Choi}
\author{Adam Ge}
\author{Selena Ge}
\author{Sylvia Zia Lee}
\author{Evin Liang}
\author{Rajarshi Mandal}
\author{Aika Oki}
\author{Daniel Wu}
\author{Michael Yang}
\affil{PRIMES STEP}
\author{Tanya Khovanova}
\affil{MIT}

\maketitle

\begin{abstract}
EvenQuads is a new card game that is a generalization of the SET game, where each card is characterized by three attributes, each taking four possible values. Four cards form a quad when, for each attribute, the values are the same, all different, or half and half. Given $\ell$ cards from the deck of EvenQuads, we can build an error-correcting linear binary code of length $\ell$ and Hamming distance 4. The quads correspond to codewords of weight 4. Error-correcting codes help us calculate the possible number of quads when given up to 8 cards. We also estimate the number of cards that do not contain quads for decks of different sizes. In addition, we discuss properties of error-correcting codes built on semimagic, magic, and strongly magic quad squares.
\end{abstract}

\section{Introduction}

EvenQuads is a new card game introduced by Rose and Perreira \cite{Rose} that is a generalization of the SET game. The new game was initially called SuperSET, but now it is called EvenQuads.

The SET deck consists of 81 cards, where each card is characterized by four attributes: number (1, 2, or 3 symbols), color (green, red, or purple), shading (empty, striped, or solid), and shape (oval, diamond, or squiggle). A set is formed by three cards that are either all the same or all different in each attribute. The players try to find sets among given cards, and the fastest player wins.

Similarly, the EvenQuads deck consists of $64$ cards, where each card is characterized by three attributes: number (1, 2, 3, or 4), color (red, green, yellow, or blue), and shape (square, icosahedron, circle, or spiral). A quad is formed by four cards that are either all the same, all different, or half and half in each attribute. The players try to find quads among given cards, and the fastest player wins.

One can view the EvenQuads deck as a deck of numbers from 0 to 63 inclusive. Four numbers form a quad if their bitwise XOR is 0 \cite{CragerEtAl,LAGames}.

Error-correcting codes allow to encode transmissions in such a way that when an error is introduced, the decoder can still completely recover the transmission. The first error-correcting code was introduced in 1950 by Hamming \cite{Hamming}. Since then, many different types of error-correcting codes have been invented serving different goals. This paper introduces linear binary error-correcting codes based on the cards from an EvenQuads deck.

We start with preliminary information about the EvenQuads game and error-correcting codes in Section~\ref{sec:preliminaries}. We extend the deck to any size that is a power of 2.

In Section~\ref{sec:qecc}, we build error-correcting codes using a set of EvenQuads cards. These are binary linear codes, such that their length is the number of cards, and the weights of the codewords are even. We call such codes quad codes. The number of quads in a given set of cards corresponds to the number of codewords of weight 4.

In Section~\ref{sec:la2help}, we describe the inequality satisfied by the dimension of a quad code. We prove that for any quad code of length $\ell$ and dimension $k$, we can find a set of cards in an EvenQuads-$2^n$ deck that realizes this code if and only if $ n \ge \ell - k - 1$.

In Section~\ref{sec:codesandsymmetries}, we describe symmetries of an EvenQuads deck and show that quad codes correspond to sets of cards up to affine transformations.

In Section~\ref{sec:smallnumbercards}, we calculate possible numbers of quads given up to 7 cards. For a given number of cards and quads, we calculate the smallest deck size in which we can have the given number of cards having exactly the given number of quads.

Section~\ref{sec:weighenumerator} introduces the weight enumerator of a code and uses it to find how many quads are possible given 8 cards. We show that the number of possible quads is one of 0, 1, 2, 3, 5, 6, 7, and 14.

In Section~\ref{sec:noquads}, we consider sets of cards from a given deck that do not contain a quad. We calculate the maximum size of such a set for small decks and provide bounds for larger decks.

Section~\ref{sec:squares} introduces error-correcting codes related to semimagic, magic, and strongly magic quad squares, which were defined in \cite{QuadSquares}. We calculate the smallest dimension for such codes and the weight enumerator corresponding to this smallest dimension. We also enumerate such codes in each possible dimension.

\section{Preliminaries}
\label{sec:preliminaries}

\subsection{EvenQuads}

An EvenQuads deck consists of $64 = 4^3$ cards with different objects. The cards have 3 attributes with 4 values each:
\begin{itemize}
\item Number: 1, 2, 3, or 4 symbols.
\item Color: red, green, yellow, or blue.
\item Shape: square, icosahedron, circle, or spiral.
\end{itemize}

A \textit{quad} consists of four cards so that for each attribute, the cards must be either all the same, all different, or half and half. 

In EvenQuads, the dealer lays out 12 cards from the deck facing up. Then, without touching the cards, players look for four cards that form a quad. As soon as the player finds a quad, they have to yell out ``quad'' and whoever says it first gets to take the four cards, as long as they are correct. If the four cards they chose do not form a quad, they have to put the cards back. If the cards do form a quad, then four new cards are drawn from the deck. If players cannot find a quad after a while, the dealer may lay out an additional card. The game ends when no more cards are left in the deck, and there are no possible quads among the remaining cards. The winner of the game is the one with the most quads.

We can assume, as suggested in \cite{CragerEtAl}, that each attribute takes values in the set $\{0,1,2,3\}$. Then, four cards form a quad if and only if the bitwise XOR (also called parity sum) of the values in each attribute is zero. This is equivalent to saying that each attribute takes values in $\mathbb{Z}_2^2$ and four cards form a quad if and only if for each attribute the sum of the vector values in $\mathbb{Z}_2^2$ is the zero vector. Thus, we can view our cards as vectors in $\mathbb{Z}_2^6$. For generalizations, we can consider an EvenQuads deck of size $2^n$. It corresponds to vectors in $\mathbb{Z}_2^n$. It follows that four vectors $\vec{a}$, $\vec{b}$, $\vec{c}$, and $\vec{d}$ form a quad if and only if
\[\vec{a} + \vec{b} + \vec{c} + \vec{d} = \vec{0}.\]

Consider four vectors $\vec{a}$, $\vec{b}$, $\vec{c}$, and $\vec{d}$ forming a quad. By a translation, we can assume that $\vec{a}$ is the origin. Then, from the above equation we get $\vec{d} = - \vec{b} - \vec{c} = \vec{b} + \vec{c}$. This means that vector $\vec{d}$ is in the same plane as the origin and vectors $\vec{b}$ and $\vec{c}$. Thus, four cards form a quad if and only if their endpoints belong to the same plane.

In the rest of the paper, we often use numbers from 0 to $2^n-1$ inclusive to label the cards in the EvenQuads-$2^n$ deck. Sometimes, to better visualize quads, we represent them as quaternary strings: we convert every number to base 4 and pad it with zeros on the left to reach a needed length of $\lceil \frac{n}{2} \rceil$. When $n$ is odd, the first digit in a string representing a card can only be 0 or 1. For example, a quad in the standard deck can be represented as four numbers 0, 21, 42, and 63. The cards in quaternary notation are 000, 111, 222, and 333.

Every quad is defined by three distinct cards. The fourth card in this quad can be uniquely found by bitwise XORing the three cards. In other words, the fourth card is the fourth point in the affine plane, defined by the three initial cards.

In particular, the EvenQuad-$2^n$ deck contains $\frac{\binom {2^n}{3}}{4}$ quads. Indeed, we can choose any three cards in $\binom {2^n}{3}$ ways, and these three cards always complete to a quad. Each quad is counted 4 times, so we have to divide by 4. The corresponding sequence is A016290:
\[1, 14, 140, 1240, 10416, 85344, 690880, 5559680, \ldots.\]

Given four cards from an EvenQuads-$2^n$ deck, the probability that they form a quad is $\frac{1}{2^n-3}$. This follows from the fact that any three cards from this deck can be uniquely extended into a quad. Thus, given $m$ cards from a Quads-$2^n$ deck, the expected number of quads is $\frac{\binom{m}{4}}{2^n-3}$.

\subsection{Error-correcting codes}

Consider a binary string $\boldsymbol{b}$ of length $\ell$. We can view it as a vector in $\mathbb{Z}_2^\ell$. The \textit{Hamming distance} between two strings of equal length is the number of positions at which the corresponding symbols are different. 

Suppose we choose a set of binary strings $C$ from $\mathbb{Z}_2^\ell$ such that they form a linear subspace and the Hamming distance between any two of them is at least 3. The strings in $C$ are called \textit{codewords}. Such a set forms an error-correcting code of \textit{length} $\ell$.

Suppose a sender sends a message $\boldsymbol{b}$ that is one of the binary strings in $C$. We allow for noisy transmission, where not more than 1 bit can be corrupted during the transmission. The receiver receives a binary string $\boldsymbol{b_1}$, where the Hamming distance between $\boldsymbol{b}$ and $\boldsymbol{b_1}$ is not more than 1. In addition, the Hamming distance between $\boldsymbol{b_1}$ and any other codeword in $C$ is at least 2. This way, the receiver can uniquely correct the corruption and figure out that the sent message is $\boldsymbol{b}$.

We assume that our codes are \textit{linear}, meaning that the set of codewords is closed under addition. The fact that the codewords form a subspace means that the number of codewords is a power of 2. We can define a \textit{weight} of the codeword to be the number of ones in its representation. Linearity of the code means that the smallest Hamming distance between two codewords equals the smallest weight of a nonzero codeword.

\begin{example}
Consider $C = \{000,111\}$. This is a code of length 3. Any received message corrupted by not more than one bit can be recovered by a majority vote.
\end{example}

Linear binary codes with a minimum distance of 4 are used a lot. Like Hamming codes with a minimum distance of 3, they can correct single-bit errors. But their advantage is that they also detect a double error. They cannot correct the double error though. Such codes are often known as SECDED (abbreviated from single error correction double error detection).

\section{Quads and error-correcting codes}
\label{sec:qecc}

We can analyze the number of quads in sets of cards by converting a set of cards into an error-correcting code. Suppose we are given cards $\vec{a_1}$ through $\vec{a_\ell}$ from an EvenQuads-$2^n$ deck. We construct a linear binary code $C$ of length $\ell$. We say $\boldsymbol{c} = c_1 c_2 \cdots c_\ell$ is a codeword if and only if:

\begin{itemize}
\item An even number of $c_1$ through $c_\ell$ are 1.
\item $c_1\vec{a_1}+\cdots+c_\ell\vec{a_\ell} = 0$.
\end{itemize}

By our definition, every codeword has an even weight.

As we require the codewords to have an even number of ones, the Hamming distance between any two codewords has to be even. For our codewords, the minimum distance is 4. Suppose, to the contrary, there are two codewords with a distance of 2. Then their sum is a codeword of weight 2. But weight 2 is impossible, as this would mean that we have two identical cards among our cards.

\begin{example}
\label{ex:ecc4}
Suppose we have four cards $\vec{a_1}$ through $\vec{a_4}$, then the length of the code $\ell$ is 4. If the cards form a quad, the only codewords are 0000 and 1111. The weight of the first codeword is 0, and the weight of the second codeword is 4.
\end{example}

It is important to note that $c_i$s are scalars and $\vec{a_i}$s are vectors. The cards are in $\mathbb{Z}_2^n$, while the codeword $\boldsymbol{b}$ can also be viewed as a vector in the space $\mathbb{Z}_2^\ell$. We use the vector symbol for the vectors corresponding to cards and bold symbols for the vectors corresponding to codewords to emphasize that they belong to different spaces.

Preliminary observations:
\begin{itemize}
\item The word with all zeros is always a codeword.
\item If we have a quad among the cards, then the corresponding word with four ones is a codeword.
\item We cannot have a codeword with two ones.
\end{itemize}

\begin{definition}
When we get code $C$ from a sequence of cards $\vec{a_1},\ldots,\vec{a_\ell}$, we say that the sequence of cards \textit{realizes} code $C$. When the cards belong to the EvenQuads-$2^n$ deck, we say that the code $C$ \textit{can be realizable} in the deck of size $2^n$.
\end{definition}

\section{Linear Algebra to Help}
\label{sec:la2help}

\subsection{From cards to codes}

We can represent the cards as an $n$-by-$\ell$ matrix $A$, consisting of zeros and ones. Each column is a card, and the $i$-th row is a set of values of the $i$th coordinate in each card. Then, the codewords are the vectors belonging to the nullspace of the matrix $A$. If we denote transpose by $T$, then the codeword $\boldsymbol{c}$ satisfies the equation
\[A \boldsymbol{c}^T = 0.\]
The condition that a codeword $\boldsymbol{c}$ has to have an even number of ones is equivalent to $\boldsymbol{c}$ being orthogonal to the vector consisting of all ones in $\mathbb{Z}_2^\ell$. Let us denote by $\boldsymbol{1}$ the vector consisting of all 1s and by $\boldsymbol{0}$ the vector consisting of all 0s in $\mathbb{Z}_2^\ell$.

For every vector $\vec{a_i}$, where $1 \leq i \leq \ell$, let us consider the augmented vector $\vec{a_i}'$, where we add one more coordinate with the value 1.

Let us consider matrix $A'$, where we add a row of all ones to the bottom of matrix $A$. In other words, the columns of $A'$ are formed by augmented vectors $\vec{a_i}'$. Then the code $C$ is exactly the nullspace of $A'$. This is because the vectors orthogonal to $\boldsymbol{1}$ are exactly the vectors that contain an even number of ones. In other words, the space of codewords $C$ is the dual space to the space of rows of $A'$. Sometimes, we denote our code $C$ as $C_A$ to emphasize that it is built by the set of cards corresponding to the matrix $A$.

Let $\boldsymbol{u}$ and $\boldsymbol{v}$ be two vectors in $\mathbb{Z}_2^\ell$. The dot product $\boldsymbol{u} \cdot \boldsymbol{v}$ is the parity of the number of places where $\boldsymbol{u}$ and $\boldsymbol{v}$ are both 1. Let $C$ be a code. Then the \textit{dual code} $C^*$ is the set of vectors $\boldsymbol{u}$ such that $\boldsymbol{u} \cdot \boldsymbol{v} =0$ for all vectors $\boldsymbol{v}$ of $C$. From standard linear algebra \cite{Axler}, we get the following proposition.

\begin{proposition}
Let $C$ be a linear binary code of length $\ell$ that contains $2^k$ codewords. Then $C^*$ is a linear binary code with $2^{\ell-k}$ codewords. Furthermore, $C^{**}=C$.
\end{proposition}

It follows that given the code $C_A$, its dual code is a span of rows of $A'$.

Note that a linear binary code has the Hamming distance of at least 4 between codewords if and only if the minimum number of ones in a nonzero codeword is at least 4. Let us call a linear binary code a \textit{quad} code if each codeword has an even number of ones, and the minimum number of ones in a nonzero codeword is at least 4. In particular, code $C_A$ described above is a quad code.

\subsection{From codes to cards}

Suppose we have a quad code $C$ of length $\ell$ and dimension $k$. We want to construct a set of cards corresponding to it. We start with the known connection between the number of codewords and code length \cite{Hamming}.

\begin{lemma}
\label{lemma:cwbound}
Let $C$ be a binary code of length $\ell$ with a minimal Hamming distance of at least 4. Then $C$ contains at most $\frac{2^{\ell-1}}{\ell}$ codewords.
\end{lemma}

In other words,
\[k \leq \ell - 1 - \lceil \log_2 \ell \rceil.\]

We are ready to prove our theorem that describes when we can find cards that realize the given code.

\begin{theorem}
\label{thm:fromCodes2Cards}
Let $C$ be a quad code of length $\ell$ containing $2^k$ codewords, where $\ell> 0$. Then, the code is realizable in an EvenQuads-$2^n$ deck if and only if
\[n \geq \ell-k-1.\]
\end{theorem}

\begin{proof}
We start with the dual space $C^*$. This space contains vector $\boldsymbol{1}$ and has dimension $\ell-k$. We pick $\ell - k - 1$ independent vectors that together with $\boldsymbol{1}$ span $C^*$.

Suppose $n < \ell - k - 1$, then the matrix $A'$ for any set of cards has fewer than $\ell -k$ rows. Thus, its nullspace has a dimension greater than $k$, implying that $C$ cannot equal its nullspace.
 
For the other direction, suppose $n \geq \ell-k-1$, then in addition $n \ge \lceil \log_2 \ell \rceil$ by Lemma~\ref{lemma:cwbound}. That means we can find $n$ vectors that span $C^*$ by adding more vectors, which can be repeat vectors, to the ones we found. We arrange these $n$ vectors as rows of matrix $A$. The columns of this matrix are our cards. It is left to show that the columns are different. It follows from the fact that $C$ is the nullspace of matrix $A'$. If the set of cards had identical cards, then $C$ would have contained a codeword of weight 2. As, by our assumption, this is not the case, all the cards must be different.
\end{proof}

\section{Codes and Symmetries}
\label{sec:codesandsymmetries}

There is a natural approach to symmetries of quads using linear algebra. So far we looked at cards as elements of a vector space $\mathbb{Z}_2^n$. A more conceptual way to do this is to see them as elements of the corresponding affine space (which is the same as a vector space, but without choosing the origin). This view emphasizes that the cards are equivalent to each other. The symmetries of the space are affine transformations: they consist of parallel translations and invertible linear transformations.

We can view an affine transformation as a pair $(M,\vec{t})$, where $M$ is an invertible $n$-by-$n$ matrix over the field $\mathbb{F}_2$ and $\vec{t}$ is a translation vector in $\mathbb{Z}_2^n$. The pair acts on vector $\vec{a}$ as $\vec{y} = M\vec{a} + \vec{t}$. We can simplify this description using augmented vectors. Indeed, $\vec{y} = M\vec{a} + \vec{t}$ is equivalent to the following
\[\begin{bmatrix} \vec{y}^T \\1\end{bmatrix} =
\left[\begin{array}{ccc|c}&M&& \vec{t}^T \\0&\cdots &0&1\end{array}\right]\begin{bmatrix} \vec{a}^T \\1 \end{bmatrix}.\]
Suppose $\vec{a'}$, $\vec{t'}$, and $\vec{y'}$ are augmented vectors for $\vec{a}$, $\vec{t}$, and $\vec{y}$. Then, we can express the affine transformation as
\[\vec{y'}^T = P \vec{a'}^T,\]
where $P$ is an invertible $(n+1)$-by-$(n+1)$ matrix above that we got from $M$ by adding a zero row at the bottom and then attaching $\vec{t'}^T$ as the last column.

Affine transformations preserve quads. Indeed if $\vec{a_i}$, for $1 \leq i \leq 4$ form a quad then the transformation produces four vectors $\vec{b_i} = M\vec{a_i} + \vec{t}$. We have
\[\sum_{i=1}^4\vec{b_i} = \sum_{i=1}^4(M\vec{a_i} + \vec{t}) = \sum_{i=1}^4 M\vec{a_i} + \sum_{i=1}^4\vec{t} = M\sum_{i=1}^4\vec{a_i} + 4\vec{t} = 0,\]
proving that vectors $\vec{b_i}$ form a quad too.

\begin{theorem}
\label{thm:equivalence}
If we change a set of cards according to an affine transformation, the corresponding code does not change. Moreover, if two sets of cards correspond to the same code, then one of the sets can be achieved from the other by an affine transformation.
\end{theorem}

\begin{proof}
Consider $\ell$ cards $\vec{a_i}$, for $1 \leq i \leq \ell$, and the corresponding matrix $A'$ and code $C_A$, where $C_A$ is the nullspace of $A'$. Applying an affine transformation $(M,\vec{t})$ to vectors $\vec{a_i}$, we get vectors $\vec{b_i} = M\vec{a_i} + \vec{t}$. We denote the new augmented matrix as $B'$. We know that $B' = PA'$. Thus, the nullspaces of $A'$ and $B'$ are the same. It follows that the new code $C_B$ is the code $C_A$ before the transformation.

For the second part, consider two sets of cards corresponding to the same code $C$. That means the augmented matrices $A'$ and $B'$ for both sets have the same nullspace $C$. It follows that the rows of $A'$ and $B'$ span the same space. That means we can represent the rows of $A'$ as a linear combination of rows of $B'$. Thus, there exists a matrix $M$, such that $A' = PB'$. Moreover, the last rows of $A'$ and $B'$ are the same. Thus, we can find $P$, such that all the elements of the last row are zero, except for the last element, which is 1. Therefore, $P$ represents the affine transformation in our space.
\end{proof}

It follows that if two sets of cards correspond to the same code, the number of quads in each set is the same.

\section{Possible number of quads for a small number of cards}
\label{sec:smallnumbercards}

We are interested in the number of possible quads given $\ell$ cards. Suppose that there exists a set of cards that produces $q$ quads with $\ell$ cards from an EvenQuads-$2^n$ deck. By adding more cards to the deck, we can conclude that any larger deck can also have $\ell$ cards with $q$ quads. Thus, it is enough to find the smallest such deck. Denote by $D(\ell,q)$ the smallest deck size $2^n$ such that there exist $\ell$ cards from the EvenQuads-$2^n$ deck that form exactly $q$ quads. If such a deck does not exist, we define $D(\ell,q)$ as $\infty$.

Number $A(\ell,4)$ describes the maximum number of codewords in a binary linear code of length $\ell$ with the Hamming distance 4. This number is described by sequence A005864 in the OEIS \cite{OEIS}, where the sequence starts with index 1:
\[1,\ 1,\ 1,\ 2,\ 2,\ 4,\ 8,\ 16,\ 20,\ 40,\ 72,\ 144,\ \ldots.\]
As each quad gives us a codeword, and there is also an all-zero codeword, the maximum number of quads among $\ell$ cards is bounded by A005864$(\ell) - 1$. Thus, for $q$ greater than A005864$(\ell) - 1$, we have $D(\ell,q) = \infty$.

We start by considering small examples without using the machinery of error-correcting codes. We cannot have a quad if we have fewer than 4 cards. The next two examples cover cases for 4 and 5 cards.

\begin{example}
\label{ex:ecc4}
Suppose we have 4 cards. If the deck is of size 4, then we have 1 quad. If it is a bigger deck, we can have 0 quads by choosing any three cards, and then for the fourth card, we can pick any other card that avoids the quad. As A005864(4) = 2, the maximum number of quads is not more than 1, thus we have covered all cases.
\end{example}

\begin{example}
\label{ex:ecc5}
Suppose we have 5 cards. It follows that the deck size is at least 8. As A005864(5) = 2, the maximum number of quads is not more than 1. We can prove directly that more than 1 quad is not possible. Indeed, if there are two quads, their intersection has 3 cards. But 3 cards define a quad uniquely. We can pick any quad and an extra card for an example of 1 quad among 5 cards. For example, 0, 1, 2, 3, and 4. If the deck size is 8, then 5 cards always contain a quad: we leave it for the reader to check this. For a larger deck, we can use cards 0, 1, 2, 4, and 8 for an example of 5 cards without a quad.
\end{example}

Table~\ref{table:possibleNumQuads} shows $D(\ell,q)$ for $\ell < 8$. The first column describes the number of cards, and the first row describes the number of quads. If the number of cards is $\ell$ and the number of quads is $q$, the entry corresponding to $\ell$ and $q$ shows $D(\ell,q)$. We replaced the infinity sign with an empty entry so as not to clutter the table.

\begin{proposition}
\label{prop:upto7cards}
Table~\ref{table:possibleNumQuads} shows $D(\ell,q)$ for $\ell < 8$.
\end{proposition}

\begin{table}[ht!]
\begin{center}
\begin{tabular}{|c|c|c|c|c|c|c|c|c|}
\hline
\# of Cards 	& 0	& 1            & 2            & 3            & 4 & 5            & 6 & 7            \\ \hline
1 		& 1	&              &              &              &   &              &   &              \\ \hline
2 		& 2	&              &              &              &   &              &   &              \\ \hline
3 		& 4 	&              &              &              &   &              &   &              \\ \hline
4 		& 8 	& 4 		&              &              &   &              &   &              \\ \hline
5 		& 16	& 8 		&              &              &   &              &   &              \\ \hline
6 		& 16	& 16 		&              & 8 		&   &              &   &              \\ \hline
7 		& 32	& 32		& 16		& 16		&   & 		 &   & 8 \\ \hline
\end{tabular}
\end{center}
\caption{$D(\ell,q)$ for $\ell < 8$.}
\label{table:possibleNumQuads}
\end{table}

\begin{proof}
The cases of 1 through 5 cards are discussed above. Given the number of quads, we want to find a corresponding quad code with the maximum possible number of codewords, as by Theorem~\ref{thm:fromCodes2Cards} they are realizable in smaller decks.

Suppose we have 6 cards with 0 quads. The corresponding quad code can have at most two codewords, 000000 and 111111. By Theorem~\ref{thm:fromCodes2Cards}, the code is realizable for $n = 4$.

Suppose we have 6 cards with $q$ quads, where $q >0$. Then, 111111 cannot be a codeword, as it has a distance of 2 from any codeword corresponding to a quad. Since the code is linear, the total number of codewords is a power of two, bounded by A005864(6) = 4. So, the total number of codewords, in this case, can be 2 or 4, with the corresponding number of quads being 1 or 3.

The code with 2 codewords can be 000000 and 001111. By Theorem~\ref{thm:fromCodes2Cards}, the code is realizable in the deck of size 16. The code with 4 codewords can be the following: 000000, 001111, 110011, and 111100. By Theorem~\ref{thm:fromCodes2Cards}, the code is realizable in the deck of size 8. 

Suppose we have 7 cards with $q$ quads. The number of codewords is a power of 2; and A005864(7) = 8. It means more than 7 quads is impossible. First, we show that 4, 5, or 6 quads are impossible.

We start by showing that we cannot have more than 1 codeword of weight 6. Indeed, two different codewords of length 7 and having six ones have to differ in exactly two places. That means the Hamming distance between them is 2, which is a contradiction.

Suppose we have at least 4 quads, then the total number of codewords is at least 8. As we have 0 or 1 codewords of weight 6, we have either 6 or 7 codewords of weight 4, proving that 4 or 5 quads are impossible.

Suppose we have 6 quads, meaning that we have 6 codewords of weight 4 and one codeword of weight 6. Without loss of generality, we can assume that the last digit of the codeword $\boldsymbol{c}$ of weight 6 is zero. It follows that the codewords of weight 4 have to have 1 as the last digit, as otherwise, they will be at distance 2 from the codeword $\boldsymbol{c}$. Since the codes are linear, adding two different codewords $\boldsymbol{a}_1$ and $\boldsymbol{a}_2$ of weight 4 should give us a working nonzero codeword. However, since the last digit is 1 in both codewords, the sum will have a 0 in the last digit. Thus, it must equal $c$. If there is another codeword $\boldsymbol{a}_3$ of weight 4, then again we see that $\boldsymbol{a}_1+\boldsymbol{a}_3 = \boldsymbol{c} = \boldsymbol{a}_1 + \boldsymbol{a}_2$. Thus, $\boldsymbol{a}_3 = \boldsymbol{a}_2$, creating a contradiction.

The remaining possibility is 7 quads, or, equivalently, 7 codewords of weight 4. An example of these 7 codewords is as follows: 1111000, 1100110, 0011110, 1010011, 0101011, 0110101, and 1001101. This code is realizable in a deck of size 8.

For 0 quads, the largest code contains 2 codewords: 0000000 and 01111111. By Theorem~\ref{thm:fromCodes2Cards}, it is realizable in a deck of size 32. For 1 quad, the code up to symmetries contains two codewords, 0000000 and 0001111, and cannot contain more codewords with weight 4. As we saw before, it cannot contain more than 1 codeword of length 6. Given that the total number of codewords is a power of 2, the code must contain exactly 2 codewords. By Theorem~\ref{thm:fromCodes2Cards}, the code is realizable in a deck of size 32. For 2 or 3 quads, the code needs to have 4 codewords. The code 0000000, 1111000, 0001111, and 1110111 corresponds to 2 quads, while the code 0000000, 1111000, 1100110, and 0011110 corresponds to 3 quads. By Theorem~\ref{thm:fromCodes2Cards}, these codes are realizable in a deck of size 16.
\end{proof}

Looking at Table~\ref{table:possibleNumQuads}, one can notice some patterns.

\begin{proposition}
\label{prop:nondecreasingcolumns}
Suppose that $D(\ell, q) = 2^n$ for some values of $\ell$, $q$, and $n$. Then $D(\ell+1, q)$ exists and is not greater than $2^{n+1}$.
\end{proposition}

\begin{proof}
Suppose that $D(\ell, q) = 2^n$ for some values of $\ell$, $q$, and $n$. Let us denote the set of $\ell$ cards in the deck of size $2^n$ with $q$ quads as $S$.

First proof. We add an extra card to this set $S$ with a numerical value $2^n$, forming a new set $S'$. The new card cannot form quads with any set of cards in $S$. Thus, the number of quads in set $S'$ is $q$, implying that $D(\ell+1, q) \le 2^{n+1}$.

Second proof. Consider a quad code corresponding to the set of cards $S$. We append 0 to each codeword to get a quad code of length $\ell+1$ and the same dimension as code $C$. By Theorem~\ref{thm:fromCodes2Cards}, this code is realizable in the deck of size at least $2^{n+1}$. As this code has the same number of codewords of weight 4 as code $C$, the cards that realize the code form $q$ quads, implying $D(\ell+1, q) \le 2^{n+1}$.
\end{proof}

Now, we want to provide examples of sets of cards for all nonempty cells in Table~\ref{table:possibleNumQuads}. We already described the examples for up to 5 cards. In the proof of Proposition~\ref{prop:nondecreasingcolumns}, we also showed how to build an example for an entry in the table that is twice the entry in the previous row and the same column. Thus, we only need to provide examples for 6 cards and 0 quads, 6 cards and 3 quads, 7 cards and 2 quads, and 7 cards and 7 quads. All these examples are realizable in the EvenQuads-16 deck and shown in Table~\ref{tab:cardexamples} using quaternary notation.

\begin{table}[ht!]
\centering
\begin{tabular}{|c|c|c|}
\hline
\# of cards        & \# of quads & Cards                      \\ \hline
\multirow{2}{*}{6} & 0           & 00, 01, 02, 10, 20, 33     \\ \cline{2-3} 
                   & 3           & 00, 01, 02, 10, 13, 03     \\ \hline
\multirow{2}{*}{7} & 2           & 00, 02, 10, 13, 20, 21, 32 \\ \cline{2-3} 
                   & 7           & 00, 01, 02, 03, 10, 11, 12 \\ \hline
\end{tabular}
\caption{Card examples.}
\label{tab:cardexamples}
\end{table}

\section{Weight enumerator}
\label{sec:weighenumerator}

We calculated possible numbers of quads for a small number of cards. To make more progress, we want to introduce the weight enumerator.

The \textit{weight enumerator} $W(x,y)$ of a binary code $C$ is given by a polynomial
\[W(x,y) = \sum_{\boldsymbol{c} \in C} x^{\ell-w(\boldsymbol{c})}y^{w(\boldsymbol{c})},\]
where $\ell$ is the length of the code, $w(\boldsymbol{c})$ is the weight of $\boldsymbol{c} \in C$, and the sum runs over all codewords. Recall that the length of a quad code is the number of cards that define it.

\begin{example}
Suppose four cards form a quad as in Example~\ref{ex:ecc4} above. The weight enumerator is $W(x,y) = x^4 + y^4$. Suppose we have 6 cards with $q$ quads. Then, from Proposition~\ref{prop:upto7cards} and its proof, the associated weight enumerator has to be $x^6 + qx^2y^4$, where $q$ is either 0, 1, or 3.
\end{example}

There is a powerful method that can help us with more advanced examples. Let us define the dual polynomial of $W$ by $W^*$, where
\[W^* = \frac{1}{W(1,1)}W(x+y,x-y).\]

MacWilliams's statement shows how the dual code's weight enumerator relates to the original code's weight enumerator.

\begin{lemma}[MacWilliams identity]
Let $C$ be a linear binary code and $W$ be its weight enumerator. Then the dual code $C^*$ has the weight enumerator $W^*$.
\end{lemma}

In particular, $W^*$ is a polynomial with non-negative integer coefficients. Note that $W(1,1)=|C|=2^k$ is the number of codewords in $C$ because $W(1,1)$ is the sum of the coefficients of the polynomial $W(x, y)$. The dual code $C^*$ is a linear binary code where each codeword has the same length $\ell$ as the original code.

We start with an example of 7 cards. We already covered this example in Section~\ref{sec:smallnumbercards}. However, here, we want to show how to use the weight enumerator.

\begin{example}
Suppose we have 7 cards with $q$ quads, where $q$ is 4, 5, or 6. Then, the corresponding code has to have 8 codewords, and the associated weight enumerator has to be $x^7 + qx^3y^4 + (7 -q)xy^6$. Then, the dual weight enumerator is 
\[\frac{1}{8}((x+y)^7 + q(x+y)^3(x-y)^4 + (7 -q)(x+y)(x-y)^6).\]
Ignoring the multiplier $\frac{1}{8}$ we get
\[8 x^7 + (4q- 28) x^6 y + (84 - 12q) x^5 y^2 + 
 8 q x^4 y^3 + 8 q x^3 y^4 + (84 - 12q)  x^2 y^5 + (4q- 28)  x y^6+ 8 y^7.\]
 The coefficient $4q-28$ is negative for given values of $q$. Therefore, such codes cannot exist.
\end{example}

\begin{proposition}
For 8 cards, the possible numbers of quads are 0, 1, 2, 3, 5, 6, 7, and 14, realizable in the smallest decks of sizes 64, 32, 32, 32, 16, 16, 16, and 8, correspondingly.
\end{proposition}

\begin{proof}
We already know that 0, 1, 2, 3, and 7 quads are possible from Proposition~\ref{prop:nondecreasingcolumns}. We also know that the maximum number of codewords is bounded by A005864$(8)$, which is 16. The coefficient of $y^8$ in the weight enumerator is at most 1, as only one codeword of weight 8 might exist. Moreover, if 11111111 is a codeword, then there are no codewords of weight 6.

We know that the EvenQuads-8 deck has 14 quads, the maximum possible number, as every three cards complete to a quad.

Suppose the dimension of the corresponding code is 4; that is, there are 16 codewords. We consider possible weight enumerators. The enumerator  $x^8+15x^4y^4$ is impossible, as 15 quads are impossible. The enumerator $x^8+14x^4y^4+y^8$ corresponds to 14 quads, which we already discussed. We are left to investigate enumerators $x^8+qx^4y^4+(15-q)x^2y^6$, where $q$ is the number of quads. The dual weight enumerator (ignoring the multiplier $1/16$) is
\begin{multline*}
(x + y)^8 + q (x + y)^4 (x - y)^4 + (15 - q) (x + y)^2 (x - y)^6 = \\
16x^8 + (-52+4q)x^7y + (88-8q)x^6y^2 + (116-4q)x^5y^3 + (-80+16q)x^4y^4 + \\
(116-4q)x^3y^5 + (88-8q)x^2y^6 + (-52+4q)xy^7 + 16y^8.
\end{multline*}
The coefficient $-52+4q$ is negative for $q<13$, and the coefficient $88-8q$ is negative for $q>11$, so the code with such an enumerator does not exist. It follows that we cannot have 8 to 13 quads inclusive.

Suppose the dimension of the corresponding code is 3; that is, there are 8 codewords. We know that it is possible to have 7 quads with 7 cards. That means it is possible to have the same number of quads with 8 cards. We know that the dimension of the corresponding code cannot be 4; therefore, it has to be 3. It follows that 7 quads with 8 cards are realizable in the EvenQuads-16 deck.

For fewer quads, we need to check the weight enumerators
\[x^8+qx^4y^4+(7-q)x^2y^6 \quad \textrm{ and } \quad x^8+6x^4y^4+y^8.\]
Consider the dual enumerator for the first one:
\begin{multline*}
(x+y)^8+q(x+y)^4(x-y)^4+(7-q)(x+y)^2(x-y)^6 \\
=8x^8+(-20+4q)x^7y+(56-8q)x^6y^2+(84-4q)x^5y^3+(16q)x^4y^4 \\
+(84-4q)x^3y^5+(56-8q)x^2y^6+(-20+4q)xy^7+8y^8.
\end{multline*}
The coefficient $-20+4q$ is negative for $q<5$, so we cannot have a code of dimension 3 with fewer than 5 quads. Now, we show that 5 and 6 quads are possible.

We start with 5 quads. Our weight enumerator suggests that we have two codewords of weight six. Without loss of generality, we can assume that these are codewords 00111111 and 11001111. We now pick one codeword of weight 4 such that it is distance 4 from both the codewords of weight 6. We use the codeword 01010011. By linearity, the other codewords are 11110000, 01101100, 10011100, and 10100011. By Theorem~\ref{thm:fromCodes2Cards}, this code is realizable in the EvenQuads-16 deck.

The 6 quad set does have an elegant associated code, where the codewords are 
\[00000000,\ 11110000,\ 00001111,\ 11001100,\ 00110011,\ 00111100,\ 11000011,\ 11111111.\] 
Again, this code is realizable in the EvenQuads-16 deck.

Suppose we have 0 quads; then the code can contain codewords of length 6 or 8. Only one codeword of length 8 exists. In addition, if the code has two different codewords of length six, their bitwise XOR has length 2, 4, or 8. Each of these lengths is forbidden. Thus, the maximum dimension of such code is 2. By Theorem~\ref{thm:fromCodes2Cards}, this code is realizable in the EvenQuads-64 deck.

Suppose we have 1 quad. Without loss of generality, we can assume that the corresponding codeword is 11110000. We cannot have a codeword of weight 8, as otherwise 00001111 will also be a codeword, thus adding another quad. Then, if there is a weight-6 codeword, its last 4 digits must be 1. Without loss of generality, we can assume that it is 00111111. Then the only other possible weight-6 codeword is 11001111. Thus, the maximum dimension is 4. By Theorem~\ref{thm:fromCodes2Cards}, this code is realizable in the EvenQuads-32 deck.

Suppose we have 2 quads. Per our previous discussion, the dimension of the code has to be 2. One example of such code is 00000000, 00001111, 11110000, and 11111111. This code is realizable in the EvenQuads-32 deck.

Suppose we have 3 quads. Again, the code dimension has to be 2. One example of such code is 00000000, 00001111, 00110011, and 00111100. This code is realizable in the EvenQuads-32 deck.

Thus, the only possible numbers of quads in a set of 8 cards are 0, 1, 2, 3, 5, 6, 7, and 14, and they are realizable in the decks of sizes 64, 32, 32, 32, 16, 16, 16, and 8, correspondingly.
\end{proof}

\begin{example}
The example of 5 quads among 8 cards in EvenQuads-16 deck is, using quaternary notation, 00, 01, 02, 03, 10, 20, 30, and 33. The example of 6 quads among 8 cards in EvenQuads-16 deck is, using quaternary notation, 03, 11, 12, 13, 21, 30, 31, and 33.
\end{example}

The weight enumerator was very useful to exclude some cases. However, 8 cards are not too many, and it might be possible to prove the proposition above without the weight enumerator. For example, one can prove that the maximum number of codewords of weight 6 is 2. Indeed, the 2 zeros in two different codewords of weight 6 cannot overlap, as otherwise, adding them will give a codeword of weight 2. For the sake of contradiction, we assume there are 3 codewords of weight 6. Without loss of generality, assume these 3 codewords are 00111111, 11001111, and 11110011. Then, we can add the three codewords to get a codeword with 2 ones, which is a contradiction.

\section{No-quads sets}
\label{sec:noquads}

Suppose we want to find sets of cards that do not contain any quads. We call such sets \textit{no-quads} sets. These sets are similar to no-sets in the game of SET, which are also called cap sets \cite{BB}. Mathematicians are very interested in finding the maximum number of cards in a no-set. Similarly, they are interested in finding the number of cards in a no-quads set. Let $F(n)$ be the maximum number of cards in a no-quads set in the EvenQuads-$2^n$ deck. For the standard deck, it was shown in \cite{CragerEtAl} that $F(6) = 9$.

Translating this to codes, we can say that we are looking for quad codes of length $\ell$ that do not contain codewords of length 4. In other words, we are looking for quad codes with a minimum distance of 6. A standard technique gives a bijection between quad codes of length $\ell$ and distance at least 6 and linear binary codes of length $\ell-1$ and distance at least 5. We can remove the last digit from a quad code $C$ to get a linear code of length $\ell-1$ where the distance decreases by not more than one. On the other hand, given a linear binary code $D$ of length $\ell-1$ and smallest positive weight 5, we can obtain a quad code by appending a parity-check bit to the end of every codeword in $D$. Thus, our task is equivalent to finding linear binary codes of length $\ell -1$ with a minimum distance of 5.

\subsection{Small decks}

Let $B(\ell)$ be the largest possible number of codewords of length $\ell$ in a linear binary code with a minimal distance of at least 6. Table~\ref{table:b} shows the values of $B(\ell)$ for $\ell \le 12$, which we manually calculated.

\begin{table}[ht!]
\begin{center}
\begin{tabular}{|c|c|c|c|c|c|c|c|c|c|c|c|c|}
\hline
$\ell$		& 1 & 2 & 3 & 4 & 5 & 6 & 7 & 8 & 9 & 10 & 11 & 12             \\ \hline
$B(\ell)$	& 1 & 1 & 1 & 1 & 1 & 2 & 2 & 2 & 4 & 4  & 4  & 8            \\ 
\hline
\end{tabular}
\end{center}
\caption{Values of $B(\ell)$ for small $\ell$.}
\label{table:b}
\end{table}

We give examples of codewords in $B(\ell)$ where $\ell$ is 1, 6, 9, and 12. Other examples with the same values of $B(\ell)$ can be constructed by appending zeros. When $\ell =1$, the only codeword is 0. When $\ell = 6$, the 2 codewords are 000000 and 111111. When $\ell = 9$, an example of a code with 4 codewords is 000000000, 111111000, 000111111, and 111000111. When $\ell = 12$, then $B(12) = 8$, and an example is the code with codewords 000000000000, 111111000000, 000000111111, 111111111111, 000111111000, 111000111000, 111000000111, and 000111000111.

\begin{theorem}
\label{thm:noquads}
If $\ell > 0$, we have $F(\ell-\log_2 B(\ell)-1)\ge \ell$. In addition, if $B(\ell) = B(\ell+1)$, then $F(\ell-\log_2 B(\ell)-1) = \ell$.
\end{theorem}
\begin{proof}
Consider a quad code $C$ of length $\ell$ with $B(\ell)$ codewords. Theorem~\ref{thm:fromCodes2Cards} implies there is a set of cards in the EvenQuads-$2^{\ell-\log_2 B(\ell)-1}$ deck that realizes $C$. Since all codewords of $C$ have at least 6 ones, this set of cards must be a no-quad set. Since $C$ has length $\ell$, this set contains $\ell$ cards, proving the first part of the theorem.

Assume to the contrary that $B(\ell) = B(\ell + 1)$ but $F(\ell-\log_2 B(\ell)-1) > \ell$, i.e. the maximum number of cards from a deck with size $2^n$ in a no-quads set has more than $\ell$ cards where, from the definition of $F(n)$, we have $n = \ell - \log_2 B(\ell) - 1$. So, there is a code $C$ corresponding to this no-quads set that has length more than $\ell$. From the definition of $B(\ell)$, the number of codewords in $C$ is not more than $B(\ell + 1) = B(\ell)$. By Theorem 3, code $C$ is only realizable in a deck of dimension $n \ge \ell+1 - \log_2 B(\ell) - 1 = \ell - \log_2 B(\ell)$. This contradicts the fact that $n = \ell - \log_2 B(\ell) - 1$.
\end{proof}

We can use Theorem~\ref{thm:noquads} and Table~\ref{table:b} to calculate the no-quads sets for small decks. Table~\ref{table:a} shows such values, where the first row is $n$, corresponding to the deck size of $2^n$, and for the second row we choose $\ell$ so that $B(\ell) = B(\ell+1)$ and $n = \ell-\log_2 B(\ell)-1$. This allows us to calculate $F(n)$ precisely.

The lemma above confirms that $F(1)=2$ (when $\ell=2$), $F(2)=3$ (when $\ell=3$), $F(3)=4$ (when $\ell=4$), $F(4)=6$ (when $\ell=6$), $F(5)=7$ (when $\ell=7$), $F(6)=9$ (when $\ell=9$), $F(7)=10$ (when $\ell=10$), and $F(8)=12$ (when $\ell=12$).

\begin{table}[ht!]
\begin{center}
\begin{tabular}{|c|c|c|c|c|c|c|c|}
\hline
$n$			& 1 & 2 & 3 & 4 & 5 & 6 & 7            \\ \hline
$\ell$ and $F(n)$	& 2 & 3 & 4 & 6 & 7 & 9 & 10              \\ 
\hline
\end{tabular}
\end{center}
\caption{Values of $F(\ell)$ for small $\ell$.}
\label{table:a}
\end{table}

\subsection{Bounds for no-quads sets}

In \cite{CragerEtAl}, the size of a no-quads set was only calculated for a standard deck. Here, we provide bounds for any deck size using bounds from coding theory \cite{ConwaySloane}. In particular, we use the notion of a Hamming bound. Let $A_{q}(\ell,d)$ denote the maximum possible size of a $q$-ary block code $C$ of length $\ell$ and minimum Hamming distance $d$. Then, the Hamming bound is:
\[A_{q}(\ell,d) \leq \frac{q^{\ell}}{\sum _{{k=0}}^{t}{\binom {\ell}{k}}(q-1)^{k}},\]
where $t=\left\lfloor {\frac {d-1}{2}}\right\rfloor$.

In particular, this bound estimates $B(\ell)$ for larger values: $B(\ell) \le A_{2}(\ell-1,5)$, which allows us to get a bound for any sized deck.

\begin{theorem}
If a no-quad set of $\ell$ cards exists in some deck, then the number of cards in the deck must be at least $\frac{\ell^2- \ell +2}{2}$. 
\end{theorem}

\begin{proof}
Consider a no-quads set with $\ell$ cards. It corresponds to a quad code of length $\ell$ with a minimal distance of at least 6. Delete any bit, producing a binary code of length $\ell - 1$ with a minimal distance of at least 5. In the Hamming bound above, we substitute $q=2$, $d=5$, and use $\ell-1$ for the length of the code to get
\[A_{2}(\ell-1,5)\leq \frac{2^{\ell-1}}{\sum _{{k=0}}^{2}{\binom {\ell-1}{k}}(2-1)^{k}} = \frac{2^{\ell-1}}{\binom {\ell-1}{0} + \binom {\ell-1}{1} + \binom {\ell-1}{2}} =\frac{2^{\ell}}{\ell^2 - \ell + 2}.\]
The dimension $k$ of the code satisfies $2^k \le \frac{2^\ell}{\ell^2 - \ell + 2}$ implying that $2^{\ell-k} \ge \ell^2-\ell+2$. 
Then, by Theorem~\ref{thm:fromCodes2Cards} this code can be realized in the deck with at least $2^{\ell-k-1} = 2^{\ell - 1}\cdot \frac{\ell^2-\ell+2}{2^\ell} = \frac{\ell^2- \ell +2}{2}$ cards.
\end{proof}

The values of $\frac{\ell^2-\ell+2}{2}$ for $\ell$ from 1 to 20 are
\[1,\ 2,\ 4,\ 7,\ 11,\ 16,\ 22,\ 29,\ 37,\ 46,\ 56,\ 67,\ 79,\ 92,\ 106,\ 121,\ 137,\ 154,\ 172,\ 191.\]
They start a famous sequence A000124 in the OEIS database \cite{OEIS} called the lazy caterer's numbers. The sequence describes the maximum number of pieces formed when slicing a pizza with $\ell$ straight cuts. The sequence can also be described as triangular numbers plus 1.

We summarize the lower bounds for small deck dimensions in Table~\ref{tab:lbds}, where we put the value in bold when we know the bound is tight. The value in the second row is $\lceil \log_2 \textrm{A000124}(\ell) \rceil$.

\begin{table}[ht!]
\centering
\begin{tabular}{|c|c|c|c|c|c|c|c|c|c|c|c|c|}
\hline
 $\ell$ & 4 & 5 & 6 & 7 & 8 & 9 & 10 & 11 & 12 & 13 & 14 & 15  \\ \hline
$n$ & \textbf{3} & 4 & 4 & \textbf{5} & 5 & \textbf{6} & 6 & 6 & 7 & 7 & 7 & 7 \\ \hline
\end{tabular}
\caption{Lower bounds for the deck dimension given the size of no-quads set.}
\label{tab:lbds}
\end{table}

On the other hand, we can estimate when the no-quads set of cards has to exist.

\begin{theorem}
If $\binom{\ell}{4}+3< 2^n$, there exists a no-quad set of $\ell$ cards in an EvenQuads-$2^n$ deck.
\end{theorem}

\begin{proof}
We use the probabilistic method. Take a random set of $\ell$ cards. The number of sets of four cards among these cards is $\binom{\ell}{4}$. Each has a probability of $\frac{1}{2^n-3}$ of being a quad. So the probability that a quad exists among our $\ell$ cards is at most $\frac{\binom{\ell}{4}}{2^n-3}$. Consequently, if $\binom{\ell}{4} < 2^n-3$, then there must exist a no-quad set of $k$ cards.
\end{proof}

We summarize the upper bounds for small sizes of no-quads sets in Table~\ref{tab:lbc}, where we put the value in bold when we know the bound is tight.

\begin{table}[H]
\centering
\begin{tabular}{|c|c|c|c|c|c|c|c|c|}
\hline
$n$ & 3 & 4 & 5 & 6 & 7 & 8 & 9 & 10 \\ \hline
$\ell$ & \textbf{4} & \textbf{5} & 6 & 7 & 8 & 10 & 12 & 14 \\ \hline
\end{tabular}
\caption{Upper bounds for the number of cards in a no-quads set given the deck dimension.}
\label{tab:lbc}
\end{table}

\section{Error-correcting codes and quad squares}
\label{sec:squares}

In our paper studying quad squares \cite{QuadSquares}, we defined a \textit{semimagic quad square} as a 4-by-4 square of EvenQuads cards such that each row and column forms a quad. 

We can index the 16 cards in a square by numbers from 1 to 16 going left-to-right and top-to-bottom. Then, the property that the set of cards forms a semimagic square is equivalent to the fact that the corresponding quad code contains the four codewords 
\[1111000000000000,\ 0000111100000000,\ 0000000011110000,\ 0000000000001111\]
corresponding to rows, and the four codewords
\[1000100010001000,\ 0100010001000100,\ 0010001000100010,\ 0001000100010001\]
corresponding to columns. We denote this code as  $C_s$, which is a vector space of dimension 7, as 4 codewords corresponding to rows sum to the same value as 4 codewords corresponding to columns.

Similarly, a \textit{magic quad square} is a 4-by-4 square of EvenQuads cards such that each row, column, and diagonal forms a quad. That means the corresponding code $C_m$ also must contain two codewords
\[1000010000100001 \textrm{ and } 0001001001001000\]
corresponding to diagonals. The codewords span the vector space of dimension 8. We add two vectors to $C_s$, but we have one more dependency: rows 1 and 4 plus columns 2 and 3 plus both diagonals sum to 0.

Both codes contain the codeword consisting of all ones. This means there is a bijection between codewords of weights $w$ and $16-w$. Two codewords correspond to each other if their bitwise XOR equals all ones. That means the weight enumerator has to be a symmetric function. We wrote a program to calculate the weight enumerator for both codes.

The weight enumerator for $C_s$ is
\[W_{C_s}(x, y) = x^{16} + 8x^{12}y^4 + 16x^{10}y^6 + 78x^8y^8 + 16x^6y^{10} + 8x^4y^{12} + y^{16}.\]
This weight enumerator can be calculated manually, too. For example, the codewords of weight 6 are in a bijection with codewords that can be formed as a bitwise XOR of two codewords that correspond to a row and a column of the square. Thus, there should be 16 of them.

The weight enumerator for $C_m$ is
\[W_{C_m}(x, y) = x^{16} + 12x^{12}y^4 + 64x^{10}y^6 + 102x^8y^8 + 64x^6y^{10} + 12x^4y^{12} + y^{16}.\]
That means every magic quad square contains 12 quads: 4 rows, 4 columns, 2 diagonals, and 2 more. For the two extra quads, we observe that
$1000010000100001+0100010001000100+0000000011110000+0001000100010001+1111000000000000
=001000110000100$. Thus, 001000110000100 is a quad, and by symmetry, 0100100000010010 is too. We can visualize them in the square as broken diagonals in Table~\ref{tab:extraquads}, where X marks cards for one quad and Y for the other. 

\begin{table}[ht!]
\centering
\begin{tabular}{|c|c|c|c|}\hline
  & X & Y &   \\\hline
X &   &   & Y \\\hline
Y &   &   & X \\\hline
  & Y & X &   \\\hline
\end{tabular}
\caption{Two extra quads in a magic square.}
\label{tab:extraquads}
\end{table}

In \cite{QuadSquares}, a \textit{strongly magic quad square} was defined to be a square such that for any four cards, if their $x$-coordinates are all the same, half and half, or all different, and their $y$-coordinates are all the same, half and half, or all different, then the four cards form a quad. Equivalently, if the coordinates of four cards form a quad, the cards form a quad. We denote this code as $C_{sm}$. Its dimension is 11 with the weight enumerator
\[W_{C_{sm}}(x, y) = x^{16} + 140x^{12}y^4 + 448x^{10}y^6 + 870x^8y^8 + 448x^6y^{10} + 140x^4y^{12} + y^{16}.\]
The code $C_{sm}$ is special: If we drop any particular bit from all codewords, we will get a code of length 15 with Hamming distance 3 and dimension 11. That means the result is a perfect code. In particular, it is the Hamming(15,11) code.

\begin{theorem}
For each type of quad square (semimagic, magic, strongly magic), Table~\ref{table:dimensions} shows how many corresponding codes are in each dimension and the smallest size of the deck of cards that realizes it.
\end{theorem}

\begin{table}[ht!]
\begin{center}
\begin{tabular}{|c|c|c|c|c|c|}
\hline
Dimension 		& 7	& 8           & 9            & 10           & 11             \\ \hline
Semimagic		& 1	& 159         & 2531         & 2823         & 112            \\ \hline
Magic 			& 	& 1           & 43           & 85           & 10              \\ \hline
Strongly magic		&  	&             &              &              & 1         \\ \hline
Deck size		& 256	& 128         & 64           & 32           & 16        \\ \hline
\end{tabular}
\end{center}
\caption{The number of codes in each dimension.}
\label{table:dimensions}
\end{table}

\begin{proof}
From Theorem~\ref{thm:fromCodes2Cards}, we know that each quad code is realizable in some deck. Moreover, from Theorem~\ref{thm:equivalence}, we know that the cards corresponding to a code are defined up to an affine transformation.

In \cite{QuadSquares}, it was shown that in the EvenQuads-$2^n$ deck, there are
\begin{align*}
2^n(2^n-1)&(2^n-2)(2^n-4)(2^n-8)\cdot \\
& (112 + 2823  (2^n - 16) + 2531 (2^n - 16) (2^n - 32) + 159 (2^n - 16) (2^n - 32) (2^n-64) + \\
  &(2^n - 16) (2^n - 32)(2^n-64)(2^n-128))
\end{align*}
semimagic quad squares. 
The proof showed that this formula is a sum of the number of equivalency classes of sets of cards forming a magic square and spanning a particular deck size times the number of elements in such a class. From here, the coefficients in the sum above correspond to the number of quad codes defining semimagic squares in different dimensions. 

Similar, in \cite{QuadSquares}, we showed that in the EvenQuads-$2^n$ deck, there are
\begin{align*}
2^n(2^n-1)&(2^n-2)(2^n-4)(2^n-8)\cdot \\
& (10 + 85(2^n - 16) + 43(2^n - 16) (2^n - 32) + (2^n - 16) (2^n - 32) (2^n-64))
\end{align*}
magic quad squares. 
The third row of the table is extracted from the formula in the same way as above.

For the fourth row, we use the theorem from \cite{QuadSquares} that shows that there are $2^n(2^n-1)(2^n-2)(2^n-4)(2^n-8)$ strongly magic quad squares that can be made by using the cards from the EvenQuads-$2^n$ deck.

From Theorem~\ref{thm:fromCodes2Cards}, we can find the smallest $n$, such that the deck size realizing a quad code of dimension $k$ corresponding to one of the squares is $2^n$. Given that the length of our codes is 16, such $n$ has to be $16 - k - 1 = 15 - k$.
\end{proof}

\section{Acknowledgments}

We are grateful to the MIT PRIMES STEP program and its director, Slava Gerovitch, for allowing us the opportunity to do this research.

\end{document}